\documentclass[reqno, 11pt]{amsart}
\usepackage[colorlinks]{hyperref}
\usepackage{url}
\usepackage{times}

\newtheorem{theorem}{Theorem}[section]

\newtheorem{proposition}[theorem]{Proposition}
\newtheorem{corollary}[theorem]{Corollary}

\theoremstyle{definition}

\newtheorem{remark}[theorem]{Remark}

\numberwithin{equation}{section}

\usepackage{bbm}
\usepackage{url}
\usepackage{euscript}
\usepackage{pb-diagram}
\usepackage{lamsarrow}
\usepackage{pb-lams}
\usepackage{amsmath}
\usepackage{amsthm}

\usepackage{epsfig}
\usepackage{amssymb}
\usepackage{pifont}
\usepackage{amsbsy}

\usepackage{graphicx}
\usepackage{epstopdf}

\newskip\aline \newskip\halfaline
\def\skipaline{\vskip\aline}

\def\qedbox{$\rlap{$\sqcap$}\sqcup$}
\def\qed{\nobreak\hfill\penalty250 \hbox{}\nobreak\hfill\qedbox\skipaline}

\def\proofend{\eqno{\mbox{\qedbox}}}

\newcommand\bC{{\mathbb C}}

\newcommand{\bN}{{{\mathbb N}}}

\newcommand\bR{{\mathbb R}}

\newcommand{\bsR}{\boldsymbol{R}}

\newcommand{\bsT}{{\boldsymbol{T}}}
\newcommand{\bsU}{\boldsymbol{U}}

\newcommand{\bsi}{\boldsymbol{\sigma}}

\newcommand{\eO}{\EuScript{O}}

\newcommand{\eQ}{\EuScript{Q}}

\newcommand{\ra}{\rightarrow}

\newcommand{\Llra}{{\Longleftrightarrow}}

\def\inpr{\mathbin{\hbox to 6pt{\vrule height0.4pt width5pt depth0pt \kern-.4pt \vrule height6pt width0.4pt depth0pt\hss}}}

\DeclareMathOperator{\su}{\boldsymbol{S}}
\DeclareMathOperator{\Seq}{\mathbf{Seq}}
\DeclareMathOperator {\Gen}{\boldsymbol{G}}

\begin{document}

\title{Regularization of certain divergent series of polynomials}

\author{Liviu I. Nicolaescu}
\thanks{Last modified on {\today}.}

\address{Department of Mathematics, University of Notre Dame, Notre Dame, IN 46556-4618.}
\email{nicolaescu.1@nd.edu}
\urladdr{\url{http://www.nd.edu/~lnicolae/}}
\subjclass[2000]{05A10, 05A19, 11B83, 40A30, 40G05, 40G10}
\keywords{translation invariant operators, divergent series, summability.}

\begin{abstract}  We  investigate the  generalized convergence  and sums of series of the form $\sum_{n\geq 0} a_n\bsT^nP(x)$, where $P\in\bR[x]$, $a_n\in\bR$, $\forall n\geq 0$, and $\bsT:\bR[x]\ra \bR[x]$  is a linear operator that commutes with the differentiation $\frac{d}{dx}:\bR[x]\ra \bR[x]$.\end{abstract}

\maketitle

\tableofcontents

\section{The main result}
\setcounter{equation}{0}

We  consider   series of the form
\[
\sum_{n\geq 0} a_n\bsT^n P(x),
\tag{$\dag$}
\label{tag: dag}
\]
where $P\in\bR[x]$, and $\bsT:\bR[x]\ra \bR[x]$ is a  linear  operator  such that 
\[
\bsT D=D\bsT,
\tag{$\ast$}
\label{tag: inv}
\]
where $D$ is the differentiation operator  $D=\frac{d}{dx}$.  The condition (\ref{tag: inv}) is equivalent with the translation invariance of $\bsT$, i.e.,
\[
\bsT\bsU^h=\bsU^h\bsT,\;\;\forall h\in \bR,
\tag{I}
\label{tag: inv1}
\]
where $\bsU^h:\bR[x]\ra \bR[x]$ is the translation operator
\[
\bR[x]\ni p(x)\mapsto p(x+h)\in\bR[x].
\]
For simplicity we set $\bsU:=\bsU^1$. Clearly $\bsU^h\in \eO$  so  a special case of the  series (\ref{tag: dag}) is the series
\[
\sum_{n\geq 0}a_n\bsU^{nh}P(x)=\sum_{n\geq 0} a_n P(x+nh),\;\;h\in\bR,
\tag{$\ddag_h$}
\label{tag: ddag}
\]
which is typically  divergent.  

We denote by $\eO$ the $\bR$-algebra of translation invariant operators. We have a natural map
\[
\eQ: \bR[[t]]\ra \eO,\;\; \bR[[t]\ni \sum_{n\geq 0}c_n\frac{t^n}{n!}\mapsto \sum_{n\geq 0} \frac{c_n}{n!}D^n.
\]
It is known  (see \cite[Prop. 3.47]{A}) that this map   is an isomorphism of rings.   We  denote by $\bsi$ the inverse  of $\eQ$
\[
\bsi:\eO\ra \bR[[t]],\;\;\eO\ni\bsT\mapsto \bsi_\bsT\in\bR[[t]].
\]
For $\bsT\in \eO$ we will refer to the formal power series $\bsi_\bsT$ as the \emph{symbol} of the operator $\bsT$. More explicitely
\[
\bsi_\bsT(t)=\sum_{n\geq 0} \frac{c_n(\bsT)}{n!}t^n,\;\;c_n(\bsT)= (\bsT x^n)|_{x=0}\in\bR.
\]
We denote by $\bN$ the set of nonnegative integers,  and by $\Seq$ the  vector space of   real sequences, i.e., maps $a:\bN\ra \bR$. Let $\Seq^c$ the vector subspace of $\Seq$ consisting of all  convergent sequences.

 A  \emph{generalized  notion of convergence}\footnote{Hardy refers to  such a notion of convergence as convergence in some `Pickwickian' sense.} or \emph{regularization method} is a pair $\mu=({}^\mu\lim, \Seq_\mu)$, where

\begin{itemize}

\item $\Seq_\mu$ is a vector subspace of $\Seq$  containing $\Seq^*$ and, 

\item ${}^\mu\lim$ is a linear map 
\[
{}^\mu\lim:\Seq_\mu\ra \bR,\;\; \Seq_\mu \ni a\mapsto {}^\mu\lim_n a(n)\in\bR
\]
 such that for any $a\in \Seq^*$ we have
\[
{}^\mu\lim a=\lim_{n\ra \infty}a(n).
\]
\end{itemize}
The sequences in $\Seq_\mu$ are called \emph{$\mu$-convergent} and ${}^\mu\lim$ is called the  \emph{$\mu$-limit}.     To any sequence $a\in\Seq$ we associate the sequence $\su[a]$ of partial sums
\begin{equation}
\su[a](n)=\sigma_{k=0}^n a(k).
\label{eq: su}
\end{equation}
A series $\sum_{n\geq 0}a(n)$ is said to by $\mu$-convergent if the sequence $\su[a]$ is $\mu$-convergent. We set
\[
{}^\mu\sum_{n\geq 0}a(n):={}^\mu\lim_n\su[a](n).
\]
We say that ${}^\mu\sum_{n\geq 0} a(n)$ is the $\mu$-sum of the series.    The regularization method is said to be \emph{shift} invariant if it satisfies the condition 
\begin{equation}
{}^\mu\sum_{n\geq 0} a(n)= a(0)+{}^\mu\sum_{n\geq 1} a(n).
\label{eq: c}
\end{equation}
We refer to the classic \cite{H} for a large collection of regularization methods.

For $x\in \bR$ and $k\in \bN$ we set
\[
[x]_k:=\begin{cases}
\prod_{i=0}^{k-1}(x-i), & k\geq 1\\
1, &k=0,
\end{cases}, \binom{x}{k}:=\frac{[x]_k}{k!}.
\]
 We can now state the main result of this paper.

\begin{theorem}  Let $\mu$ be a regularization method, $\bsT\in \eO$ and $f(t)=\sum_{n\geq 0} a_nt^n\in\bR[[t]]$.  Set $c:=c_0(\bsT)=\bsT 1$.  Suppose that  $f$ is \emph{$\mu$-regular at $t=c$}, i.e.,
\[
\mbox{for every $k\in \bN$  the series  $\sum_{n\geq 0}a_n[n]_k\, c^{n-k}$ is $\mu$-convergent.}
\tag{$\mu$}
\label{tag: mu}
\]
We denote by $f^{(k)}(c)_\mu$ its $\mu$-sum
 \[
 f^{(k)}(c)_\mu={}^\mu\sum_{n\geq 0} a_n[n]_k\,c^{n-k}.
 \]
 Then for every  $P\in\bR[x]$ the series $\sum_{n\geq 0}a_n(\bsT^n P)(x)$ is $\mu$-convergent and its $\mu$-sum is
 \[
 {}^\mu\sum a_n(\bsT^n P)(x)=f(\bsT)_\mu P\,(x),
 \]
 where $f(\bsT)_\mu\in\eO$ is the operator
 \begin{equation}
 f(\bsT)_\mu:=\sum_{n\geq 0} \frac{f^{(k)}(c)_\mu}{k!}(\bsT-c)^k.
 \label{eq: regf}
 \end{equation}
 \label{th: main}
  \end{theorem}
  
  \begin{proof} Set $\bsR:=\bsT-c$ and let $P\in\bR[x]$. Then
  \[
  \bsR=\sum_{n\geq 1}\frac{c_n(\bsT)}{n!}D^n
  \]
  so that
  \begin{equation}
\bsR^nP=0,\;\;\forall  n>\deg P.
\label{eq: vanish}
\end{equation}
In particular this shows that $f(\bsT)_\mu$ is well defined. We have
\[
a_n\bsT^n P=a_n(c+\bsR)^nP=a_n\sum_{k=0}^n \binom{n}{k}c^{n-k} \bsR^{k} P=\sum_{k=0}^{\deg P}\binom{n}{k}c^{n-k} \bsR^{k} P.
\]
At the last step we used (\ref{eq: vanish}) and the fact that 
\[
\binom{n}{k}=0,\;\;\mbox{if}\;\; k>n.
\]
This shows that  the formal series $\sum_{n\geq 0}a_n(\bsT^n P)(x)$ can be written as a \emph{finite} linear combination of formal series
\[
\sum_{n\geq 0}a_n(\bsT^n P)(x)= \sum_{k=0}^{\deg P}\frac{\bsR^k P(x)}{k!}\left(\,\sum_{n\geq 0} a_n[n]_k\,c^{n-k}\,\right).
\]
 From the linearity of the $\mu$-summation operator we deduce 
 \[
 {}^\mu\sum_{n\geq 0}a_n(\bsT^n P)(x)= \sum_{k=0}^{\deg P}\frac{\bsR^k P(x)}{k!}\left(\,{}^\mu\sum_{n\geq 0} a_n[n]_k\,c^{n-k}\,\right)
 \]
 \[
 =\left(\sum_{k=0}^{\deg P}\frac{f^{(k)}(c)_\mu}{k!}\bsR^k\right)\,P(x)=f(\bsT)_\mu P\,(x)
  \]
  \end{proof}

\section{Some applications}
\setcounter{equation}{0}
To   describe some consequences of Theorem \ref{th: main} we need to first describe some classical facts about  regularization methods. 

For any sequence $a\in\Seq$ we denote by $\Gen_a(t)\in\bR[[t]]$ its generating series. We regard the partial sum construction $\su$ in (\ref{eq: su}) as a linear operator  $\su:\Seq\ra \Seq$. Observe that
\[
\Gen_{\su[a]}(t)=\frac{1}{1-t}\Gen_a(t).
\label{eq: gen-su}
\]
 We say that a  regularization method $\mu_1=({}^{\mu_1}\lim,\Seq_{\mu_1})$ is stronger  than  the regularization method $\mu_0=({}^{\mu_1}\lim,\Seq_{\mu_0})$, and we write this $\mu_0\prec\mu_1$, if 
\[
\Seq_{\mu_0}\subset \Seq_{\mu_1}\;\;\mbox{and}\;\;{}^{\mu_1}\lim_n a(n)={}^{\mu_0}\lim_n a(n),\;\;\forall a\in \Seq_{\mu_0}.
\]
The \emph{Abel regularization method}\footnote{This  was apparently known and used by Euler.} $A$    is defined as follows.  We say  that a sequence $a$ is  $A$ convergent if  

\begin{itemize}

\item the radius of convergence  of the series $\sum_{n\geq 0} a_n t^n$ is at least $1$ and

\item  the  function  $t\mapsto (1-t)\sum_{n\geq 0}a_nt^n$ has a finite limit as $t\ra 1^-$. 

\end{itemize}

Hence
\[
{}^A\lim a(n)=\lim_{t\ra 1^-} (1-t)\sum_{n\geq 0} a_nt^n,
\]
and $\Seq_A$ consists of sequence for which the above limit exists  and it is finite.   Using (\ref{eq: gen-su}) we deduce that a series $\sum_{n\geq 0}a(n)$ is $A$-convergent if and only if the limit
\[
\lim_{t\ra 1^-}\sum_{n\geq 0}a_nt^n
\]
exists and it is finite.

Let $k\in\bN$. A sequence $a\in\Seq$ is said to be \emph{$C_k$-convergent} (or \emph{Ces\`{a}ro convergent of order $k$})  if the limit
\[
\lim_{n\ra \infty} \frac{\su^k[a](n)}{\binom{n+k}{k}}
\]
exists and it is finite. We denote this limit by ${}^{C_k}\lim a(n)$. A  series  $\sum_{n\ geq 0} a(n)$ is said to be $C_k$-convergent  if the sequence of partial sums $\su[a]$ is $C_k$ convergent.  Thus the $C_k$-sum of this series is
\[
{}^{C_k}\sum_{n\geq 0} a(n)=\lim_{n\ra \infty}\frac{\su^{k+1}[a](n)}{\binom{n+k}{k}}.
\]
More explicitly, we have  (see \cite[Eq.(5.4.5)]{H})
\[
{}^{C_k}\sum_{n\geq 0} a(n)=\lim_{n\ra \infty}\frac{1}{\binom{n+k}{k}}\Biggl(\,\sum_{\nu=0}^n\binom{\nu+k}{k}a(n-\nu)\,\Biggr)
\]
Hence
\[
{}^{C_k}\sum_{n\geq 0} a(n)\Llra \su^{k+1}[a](n)\sim A\binom{n+k}{k}\sim A\frac{n^k}{k!},
\]
where
\[
a\sim b \Llra \lim_{n\ra\infty}\frac{a(n)}{b(n)}=1,
\]
if $a(n), b(n)\neq 0$, for $n\gg 0$.

The $C_0$ convergence is equivalent  with the classical convergence and it is known  (see \cite[Thm. 43, 55]{H})  that
\[
C_k\prec C_{k'}\prec A,\;\;\forall  k<k'.
\]
Given this fact, we define a sequence to be  $C$-convergent (Ces\`{a}ro convergent) if it is $C_k$-convergent for some  $k\in\bN$. Note that $C\prec A$.  Both the $C$ and $A$ methods are shift invariant, i.e., they satisfy the condition (\ref{eq: c}).

\begin{proposition}  The power series
\[
f(t)=(1+t)^{-1}=\sum_{n\geq 0}(-t)^n
\]
is $C$-regular at $t=1$, and
\[
f^{(k)}(1)_C=\frac{(-1)^k k!}{2^{k+1}},\;\; f(r)_C=\sum_{k\geq 0} \frac{f^{(k)}(1)_C}{k!} r^k=f(1+r)= (2+r)^{-1}.
\]
\label{prop: C-reg}
\end{proposition}

\begin{proof} It suffices  show that for any $k\geq 1$ the series 
\[
\sum_{n\geq 0}(-1)^n\binom{n}{k-1}=\frac{1}{(k-1)!}\sum_{n\geq 0}(-1)^n[n]_{k-1}
\]
 is $C_k$ convergent to $\frac{(-1)^{k-1}}{2^k}$.  For $k\geq 1$ we denote  by $\hat{\beta}_{k-1}$ the sequence
\[
\hat{\beta}_{k-1}(n)=(-1)^n\binom{n}{k-1},\;\;\forall n\geq 0.
\]
 Then
\[
\Gen_{\hat{\beta}_{k-1}}=\sum_{n\geq 0} (-1)^n\binom{n}{k-1} t^n
\]
\[
=\sum_{n\geq k-1}(-1)^n\binom{n}{k-1}t^n=(-1)^{k-1}t^{k-1}\sum_{\nu\geq 0}(-1)^{\nu}\binom{\nu+k-1}{k-1}t^\nu=(-t)^{k -1}(1-t)^{-k}.
\]
Then
\[
\Gen_{\su^{k+1}[\hat{\beta}_{k-1}]}(t)=\frac{(-t)^{k-1}}{(1+t)^k(1-t)^{k+1}}=(-1)^{k-1}\frac{t^{k-1}(1+t)}{(1-t^2)^{k+1}}
\]
\[
=(-1)^{k-1}(t^k+t^{k-1})\sum_{m\geq 0}\binom{m+k}{k}t^{2m}=(-1)^{k-1}\sum_{m\geq 0}\binom{m+k}{k}(t^{2m+k}+t^{2m+k-1}).
\]
Define
\[
\mu(n)= \begin{cases}
\frac{n-k}{2}, & n\equiv k\bmod 2\\
\frac{n-k+1}{2}, &n\equiv k+1\bmod 2.
\end{cases}
\]
Then
\[
\Gen_{\su^{k+1}[\hat{\beta}_{k-1}]}(t)=(-1)^{k-1}\sum_{n\geq k-1}\binom{\mu(n)+k}{k} t^n.
\]
The desired conclusion follows by observing that
\[
\binom{\mu(n)+k}{k}\sim \frac{1}{2^k}\cdot\frac{n^k}{k!}.
\]
\end{proof}

\begin{corollary} The  series $f(t)=\log(1+t)=\sum_{n\geq 1}(-1)^{n+1}\frac{t^n}{n}$ is $C$-regular at $t=1$ and
\[
f(r)_C=\sum_{k\geq 0} \frac{f^{(k)}(1)_C}{k!} r^k=f(1+r)=\log(2+r).
\]
\end{corollary}

\begin{proof} Clearly  the alternating series
\[
f(1)=\sum_{n\geq 1}(-1)^{n+1}\frac{1}{n}
\]
is convergent, thus $C$-convergent.   The $C$-convergence of the series
\[
\sum_{n\geq 1}(-1)^{n+1}\frac{[n]_k}{n}
\]
now follows from the previous proposition since $D_tf=(1+t)^{-1}$. 
\end{proof}

We have the following immediate result.

\begin{proposition} Suppose  that  $f(z)$ is a holomorphic function  defined in an open neighborhood of the set  $\{1\}\cup \{|z|\}\subset \bC$. If $\sum_{n\geq 0}a_nz^n$ is the Taylor series expansion of $f$ at $z=0$ then the corresponding formal power series $[f]=\sum_{n\geq 0}a_nt^n$ is $A$-regular at $t=1$,
\[
[f]^{(k)}(1)_A= f^{k}(1),
\]
and the series
\[
 [f](r)_A=\sum_{[f]^{(k)}(1)_A}{k!}r^k
 \]
coincides with   so the  Taylor expansion of $f$ at $z=1$,   it converges to $f(1+r)$.
\label{prop: A-reg}
\end{proposition}

\begin{corollary} Suppose  that  $f(z)$ is a holomorphic function  defined in an open neighborhood of the set  $\{1\}\cup \{|z|\}\subset \bC$ and $\sum_{n\geq 0}a_nz^n$ is the Taylor series expansion of $f$ at $z=0$.  Then for every $\bsT$ in $\eO$ such that $c_0(\bsT)=1$, any $P\in\bR[x]$, and any $x\in bR$ we have
\[
{}^A\sum_n a_n\bsT^n P(x)=\sum_{k\geq 0} \frac{f^{k}(1)}{k!}(\bsT-1)^kP(x).\proofend
\]
\label{cor: reg-A}
\end{corollary}

Recall that the Cauchy product of two sequences $a, b\in \Seq$ is the sequence $a\ast b$,
\[
a\ast b(n)=\sum_{i=0}^n a(n-i)b(i),\;\;\forall n\in \bN.
\]
A regularization method is said to be \emph{multiplicative} if
\[
{}^\mu\sum_n  a\ast b(n) =\left(\,  {}^\mu\sum_n  a(n)\, \right) \left(\, {}^\mu\sum_n b(n)\,\right),
\]
for any $\mu$-convergent series $\sum_{n\geq 0}a(n)$ and $\sum_{n\geq 0}b(n)$. The results  of \cite[Chap.X]{H} show that the $C$ and $A$ methods are multiplicative.

For any  regularization method $\mu$ and $c\in \bR$ we denote by $\bR[[t]]_\mu$ the set of series that are $\mu$-regular at $t=1$. 

\begin{proposition} Let $\mu$  be a multiplicative  regularization method.  Then $\bR[[t]]_\mu$ is a commutative ring with one and we have the product rule
\[
(f\cdot g)^{(n)}(1)_\mu=\sum_{k=0}^n\binom{n}{k} f^{(k)}(1)_\mu\cdot g^{(n-k)}(1)_\mu.
\]
Moreover, if $\bsT\in\eO$ is such that $c_0(T)=1$ then the map
\[
\bR[[t]]_\mu\ni f\mapsto f(\bsT)_\mu\in \eO
\]
is a ring morphism.
\label{prop: CA}
\end{proposition}

\begin{proof}  The product formula follows  from the iterated application of the equalities
\[
D_t(fg)=(D_tf)g+f(D_t g),\;\;(fg)(1)_\mu=f(1)_\mu\cdot g(1)_\mu,\;\;f'(1)_\mu=(D_tf)(1)_\mu,
\]
where $D_t:\bR[[t]]\ra \bR[[t]]$ is the formal differentiation operator $\frac{d}{dt}$. The last statement is an immediate application of the above product rule.
\end{proof}

\begin{remark} The inclusion $\bR[[t]]_C\subset \bR[[t]]_A$   is strict.  For example the power series 
\[
f(z) =e^{1/(1+z)}
\]
satisfies the assumption  of Proposition \ref{prop: A-reg} so that the associated formal power series $[f]$ is $A$-regular at $1$. On the other hand, the arguments in \cite[\S 5.12]{H} show that $[f]$ is not $C$-regular at $1$.  \qed
\end{remark}

Consider the translation operator $\bsU^h\in \eO$.  From Taylor's formula
\[
p(x+h)=\sum_{n\geq 0} \frac{h^n}{n!}D^n p(x)
\]
we deduce  that
\[
\bsi_{\bsU^h}(t)=e^{th}.
\]
Set $\Delta_h:=\bsU^h-1$.  Using Proposition \ref{prop: C-reg} and Theorem \ref{th: main} we deduce the following result.

\begin{corollary} For any $P\in\bR[x]$ we have  
\begin{equation}
{}^C\sum_{n\geq 0} (-1)^nP(x+nh)=\frac{1}{2}\left(\,\sum_{n\geq 0}\frac{(-1)^n}{2^n}\Delta_h^n\,\right)P(x).
\label{eq: alt-trans0}
\end{equation}
\label{cor: alt-trans}
\end{corollary}

Observe that
\[
\left(\,1+\frac{1}{2}\Delta_h\,\right)\left(\,\sum_{n\geq 0}\frac{(-1)^n}{2^n}\Delta_h^n\,\right)=1
\]
so that $\frac{1}{2}\sum_{n\geq 0}\frac{(-1)^n}{2^n}\Delta_h^n$ is the inverse of the operator $2+\Delta_h$. We thus have
\begin{equation}
{}^C\sum_{n\geq 0}(-1)^nP(x+nh)=(2+\Delta_h)^{-1} P(x)=(1+\bsU^h)^{-1}P(x).
\label{eq: alt-trans}
\end{equation}
The  inverse  of $1+\bsU^h$ can be explicitly expressed using  Euler numbers and polynomials, \cite[Eq. (14), p.134]{No}. The Euler numbers $E_k$ are defined by the Taylor expansion
\[
\frac{1}{\cosh t}=\frac{2}{e^t+e^{-t}}=\sum_{k\geq 0}\frac{E_k}{k!}t^k.
\]
Since $\cosh t$ is an even function we deduce that $E_k=0$ for odd $k$.  Here are the first few Euler numbers.
\[
\begin{tabular}{||c|c|c|c|c|c|c|c|c|c||}\hline
$n $  & $0$            & $2 $           & $4 $          & $6 $          & $8$             & $10$            & $12$               & $14$ & $16$  \\ \hline
      &     &                &                &               &               &                 &                 &                    &          \\
$E_n$ & $1$ & $-1$ & $5$ & $-61$ & $1,385$ & $-50,521$  & $2,702,765$ & $-199,360,981$ &  $19,391,512,145$ \\ \hline
\end{tabular}
\]

Then
\[
\frac{1}{1+\bsU^h}=\frac{ \bsU^{-\frac{h}{2} } }{ \bsU^{ \frac{h}{2} } +\bsU^{ -\frac{h}{2} } } =\frac{ \bsU^{-\frac{1}{2}} }{ e^{ \frac{D}{2} }+e^{-\frac{D}{2} } }=\frac{1}{2}\bsU^{-\frac{h}{2} }\frac{1}{ \cosh \frac{hD}{2} }=\frac{1}{2}\bsU^{-\frac{h}{2}}\sum_{k\geq 0} \frac{E_kh^k}{2^kk!}D^k.
\]
Hence
\begin{equation}
{}^C\sum_{n\geq 0} (-1)^nP(x+nh)= \frac{1}{2}\sum_{k\geq 0} \frac{E_kh^k}{2^kk!}P^{(k)}\left(\,x-\frac{h}{2}\,\right).
\label{eq: nor}
\end{equation}
When $P(x)=x^m$, $h=1$, and $x=0$ we deduce
\begin{equation}
{}^C\sum_{n\geq 0}(-1)^n n^m= \frac{1}{2^{m+1}}\sum_{k\geq 0} (-1)^{m-k}E_k\binom{m}{k}.
\label{eq: nor1}
\end{equation}
When $P(x)=\binom{x}{m}$, $x=0$, $h=1$ then it is more convenient to use  (\ref{eq: alt-trans0}) because 
\[
\Delta \binom{x}{k}=\binom{x}{k-1},\;\;\forall k,x.
\]
We deduce
\begin{equation}
{}^C\sum_{n\geq 0}(-1)^n\binom{n}{m}=\frac{1}{2}\sum_{k=1}^m\frac{(-1)^k}{2^k}\binom{0}{m-k}=\frac{(-1)^m}{2^{m+1}}.
\label{eq: nor2}
\end{equation}

\begin{remark}  Here is a more direct  (and almost complete) proof of the equality    (\ref{eq: alt-trans}) assuming the  Ces\`{a}ro convergence of  the series $\sum_{n\geq 0}(-1)^nP(x+nh)$. Denote by $S(x)$ the Ces\`{a}ro sum of this series. Then
\[
S(x+h)={}^C\sum_{n\geq 0} (-1)^n P\bigl(\, x+(n+1)h\,\bigr) \]
\[
\stackrel{ (\ref{eq: c}) }{=} \,-\,{}^C\sum_{n\geq 0} (-1)^n P(x+h) +P(x)= -S(x)+P(x).
\]
Hence
\[
S(x+h)+S(x)=P(x),\;\;\forall x\in \bR.
\]
If we knew that $S(x)$ is a polynomial we would then deduce 
\[
S(x)=(1+\bsU^h)^{-1}P(x).\proofend
\]
\end{remark}

\begin{remark} We want to comment a bit about possible methods of establishing $C$-convergence. To formulate a general strategy we need to introduce a classical notation. More precisely, if $f(t)=\sum_{n\geq 0}a_nt^n$ is a formal power series  we     let $[t^n]f(t)$ denote the coefficient of $t^n$ in this power series, i.e. $[t^n]f(t)=a_n$.

  Let $f(t)=\sum_{n\geq 0} a_nt^n$.   Then  the series   $\sum_{n\geq 0}a_nt^n$  $C$-converges to $A$ if and only if there exists   a  nonnegative  \emph{real} number $\alpha$ such that
  \[
  [t^n]\left( (1-t)^{-(\alpha+1)} f(t)\,\right)\sim A \frac{n^\alpha}{\Gamma(\alpha+1)}
  \]
  where $\Gamma$ is Euler's  Gamma function. For a proof we refer  to \cite[Thm. 43]{H}.  This characterization leads to the following

  \medskip
  
 \noindent \textbf{Ces\`{a}ro  summability  meta-principle.}  {\it Suppose   that the power series  $f(t)=\sum_{n\geq 0} a_nt^n$ defines a holomorphic function $f(z)$  such that  the following  hold.
  
  \begin{itemize} 
  
  \item The domain of $f(z)$ contains  the open disk $\{|z|<1\}$.
  
  \item  Along the  unit  circle $\{|z|=1\}$ the function $f(z)$  has only   finitely many singular points
  \[
  \zeta_1,\dotsc, \zeta_\nu\neq 1.
  \]
  \item  For every singular  point $\zeta_k$ there exists a positive integer $m_k$ such that 
  \[
  \lim_{z\ra \zeta_k}|z-\zeta_k|^{m_k} f(z)=0.
  \]
  \end{itemize} 
  
  Then the series $\sum_{n\geq 0} a_n$ is Ces\`{a}ro summable and the Ces\`{a}ro sum  equals the   Abel sum, $\lim_{t\nearrow 1} f(t)$. }
  
  \medskip

 For a more detailed description of the conditions when this meta-principle   is a genuine   principle we refer to \cite[Thm. VI.5]{FS}.  \qed
 
 \end{remark}

\end{document}